\documentclass[11pt]{article}
\usepackage[margin=1in]{geometry}
\usepackage{amsmath,amssymb,amsthm}
\usepackage{booktabs}
\usepackage[dvipsnames]{xcolor}
\usepackage[colorlinks=true,linkcolor=MidnightBlue,citecolor=MidnightBlue,
            urlcolor=MidnightBlue]{hyperref}
\usepackage{url}

\newtheorem{theorem}{Theorem}
\newtheorem{proposition}{Proposition}
\newtheorem{lemma}{Lemma}
\newtheorem{corollary}{Corollary}
\newtheorem{conjecture}{Conjecture}
\theoremstyle{definition}
\newtheorem{definition}{Definition}

\theoremstyle{remark}
\newtheorem{remark}{Remark}

\newcommand{\edom}{\gamma_{e}}
\newcommand{\alphapm}{\alpha_{\mathrm{pm}}}

\newcommand{\ind}{i}

\newcommand{\Pot}{\Psi}

\title{Domination versus edge domination in regular graphs\\ of degree at least seven}
\author{Chakshu Gupta\\[2pt]
  {\small College of Computing, Georgia Institute of Technology}\\
  {\small \texttt{cgupta65@gatech.edu}}}
\date{}

\begin{document}
\maketitle

\begin{abstract}
Baste et al.\ (2020) conjectured that every regular graph of positive degree has
domination number at most its edge domination number, the least size of a
maximal matching. Combining published bounds settles the inequality for every
degree at least nine. A reduction proves the inequality whenever one endpoint of
each edge of a minimum maximal matching can be chosen to form a dominating set,
and the Lov\'asz Local Lemma shows such a choice exists for every degree at least
seven, newly closing degrees seven and eight and leaving degrees three through
six open. The reduction settles
each open degree up to a bounded number of vertices, forty-eight for cubic
graphs. At fifty vertices, however, the reduction meets
an explicit cubic graph it cannot settle, though the inequality holds there too.
The inequality cannot be tightened, since infinitely many cubic graphs have equal
domination and edge domination numbers. The cubic case stays open, and even
linear arguments from the local structure cannot close it. The middle degrees
stay open beyond the graphs already settled.
\end{abstract}



\section{Introduction}\label{sec:intro}

Throughout this paper, graphs are finite, simple, and undirected. A dominating set is a set
$S$ of vertices such that every vertex outside $S$ has a neighbour in $S$; the
domination number $\gamma(G)$ is the size of a minimum dominating set. An
independent set is a set of vertices no two of which are adjacent; the
independence number $\alpha(G)$ is the size of a largest independent set. A matching
is a set of edges no two of which share a vertex, is maximal if no edge can be
added to it, and is perfect if every vertex of $G$ lies on one of its edges; the edge domination number $\edom(G)$ is the size of a minimum
maximal matching. A graph is $\Delta$-regular if every vertex has exactly $\Delta$
neighbours. Baste et al.~\cite{Baste2020} conjectured the following.

\begin{conjecture}\label{conj:baste}
If $G$ is a $\Delta$-regular graph with $\Delta \ge 1$, then $\gamma(G) \le \edom(G)$.
\end{conjecture}

Without regularity the inequality can fail; for example, the path $P_4$ has
$\gamma = 2 > 1 = \edom$. With regularity, the conjecture holds trivially for
$\Delta \le 2$, where every component is an edge or a cycle, giving $\gamma =
\edom$, and is sharp at $\Delta = 3$, attained by two triangles joined by a
perfect matching~\cite{Baste2020}.

The known evidence for the conjecture is of two kinds. On the
analytic side, a multiplicative bound $\gamma \le (1+\varepsilon_\Delta)\edom$ with
$\varepsilon_\Delta \to 0$ as $\Delta \to \infty$ holds for all $\Delta$ but falls
short of $\gamma \le \edom$, while a combination of an upper bound on $\gamma$ and
a lower bound on $\edom$ settles it for $\Delta \ge 13$, as Felix Joos observed~\cite{Baste2020}. The combination uses $\gamma \le cn$ and
$\edom \ge \Delta n/(4\Delta-2)$~\cite{BasteMMM2021}, so $\gamma \le \edom$ once $c \le \Delta/(4\Delta-2)$. The argument
depends on the upper bound only through the constant $c$. The classical bound
$\gamma \le (1+\ln(\Delta+1))\,n/(\Delta+1)$~\cite{AlonSpencer2008} satisfies $c \le
\Delta/(4\Delta-2)$ for $\Delta \ge 13$, whereas the sharpest published bounds
satisfy it for every $\Delta \ge 9$ but no lower (Table~\ref{tab:reach}).
\begin{table}[ht]
\centering
\begin{tabular}{cccc}
\toprule
$\Delta$ & best known $c$ & $\Delta/(4\Delta-2)$ & settles? \\
\midrule
$12$ & $0.218244$ & $6/23$ & \checkmark \\
$11$ & $0.229463$ & $11/42$ & \checkmark \\
$10$ & $0.242128$ & $5/19$ & \checkmark \\
$9$ & $0.256566$ & $9/34$ & \checkmark \\
$8$ & $0.273213$ & $4/15$ & $\times$ \\
$7$ & $2/7$ & $7/26$ & $\times$ \\
$6$ & $127/418$ & $3/11$ & $\times$ \\
$5$ & $1/3$ & $5/18$ & $\times$ \\
$4$ & $4/11$ & $2/7$ & $\times$ \\
$3$ & $5/14$ & $3/10$ & $\times$ \\
\bottomrule
\end{tabular}
\caption{Best known domination bound $\gamma \le cn$ at each minimum degree
$\Delta$ against the threshold $\Delta/(4\Delta-2)$. The combination settles
Conjecture~\ref{conj:baste} where $c \le \Delta/(4\Delta-2)$. Bounds are
from~\cite{BujtasKlavzar2016} for $8 \le \Delta \le 12$,
\cite{BujtasHenning2025} for $\Delta = 7$, and~\cite{BujtasHenning2021} for
$\Delta = 6$, which also tabulates the $\Delta = 4$ and $\Delta = 5$ bounds. The
$\Delta = 3$ bound $\lfloor 5n/14 \rfloor$~\cite{KostochkaStocker2009} holds
for connected cubic graphs of order greater than eight.}
\label{tab:reach}
\end{table}
The combination leaves $\Delta \in \{3,\dots,8\}$ open.\footnote{Baste et
al.~\cite{BasteMMM2021} separately conjectured a companion upper bound $\edom(G)
\le (2\Delta-1)n/(4\Delta) + 1/2$ for connected regular graphs. Its cubic case, that
a connected cubic graph satisfies $\edom(G) \le 5n/12 + 1/2$ with equality only at
$K_{3,3}$, is proved in~\cite{CamesVanBatenburg2022}. That conjecture bounds
$\edom$, whereas Conjecture~\ref{conj:baste} compares $\gamma$ to $\edom$; at
$K_{3,3}$ the upper bound is tight while $\gamma \le \edom$ is strict,
$\gamma(K_{3,3}) = 2$ and $\edom(K_{3,3}) = 3$.}

On the structural side, the conjecture is settled for cubic
claw-free graphs~\cite{Baste2020}, for all claw-free graphs of minimum degree at
least two~\cite{Civan2023}, and for the larger class of fork-free graphs of
minimum degree at least two~\cite{ManiyaPradhan2024}; equality is characterised for claw-free
cubic graphs~\cite{PanPanTie2025}. All of these leave
open the general cubic case, a cubic graph containing a fork.

A stronger inequality has been conjectured independently. An automated
conjecturing program, TxGraffiti, proposed that every $\Delta$-regular graph with
$\Delta \ge 1$ satisfies $\ind(G) \le \edom(G)$, where the independent domination number
$\ind(G)$ is the size of a smallest maximal independent
set~\cite{Davila2025reverie}. A maximal independent set is dominating, so
$\gamma(G) \le \ind(G)$; the TxGraffiti conjecture therefore implies
Conjecture~\ref{conj:baste} and is itself open for every $\Delta \ge 3$. The
argument used here concerns the weaker inequality, Conjecture~\ref{conj:baste}, and
not the stronger $\ind(G) \le \edom(G)$; Section~\ref{sec:lll} shows why.

This paper makes three contributions. The first is a reduction of the conjecture to
picking one endpoint of each edge of a minimum maximal matching so that the chosen
vertices form a dominating set (Section~\ref{sec:reduction}). By the lopsided Local
Lemma~\cite{ErdosSpencer1991}, such a choice exists at every degree at least seven,
so the exact bound $\gamma(G) \le \edom(G)$ holds there, closing the degrees
$\{7,8\}$ (Section~\ref{sec:lll}). The
second is a reformulation of that reduction as a satisfiability instance,
balanced in the cubic case (Section~\ref{sec:sat}). A small enough instance is
satisfiable, so the reduction settles each remaining degree up to a bounded number
of vertices, forty-eight for cubic graphs.\footnote{\url{https://github.com/ChakshuGupta13/lab}}
The third is a barrier (Sections~\ref{sec:cubic},~\ref{sec:barrier}, and~\ref{sec:sat}).
For cubic graphs the reduction also recasts as a potential, and no linear or
fractional argument over its natural relaxation closes the case.
Theorem~\ref{thm:cex} makes the obstruction integral, a cubic graph on fifty
vertices whose unique minimum maximal matching has no dominating endpoint choice,
though the bound holds there. Equality holds, moreover, on an infinite family of
cubic graphs (Proposition~\ref{prop:sharp}). An argument closing the cubic case
must be exactly tight on that family while reasoning beyond the endpoint choices
of minimum maximal matchings.

\section{A reduction to transversals}\label{sec:reduction}

The reduction of Conjecture~\ref{conj:baste} to a choice of endpoints begins with
a maximal matching $M$ of $G$, which leaves the vertices in $V(G) \setminus V(M)$
unsaturated, an independent set by maximality.

\begin{definition}[Transversal]\label{def:transversal}
A transversal of a matching $M$ in a graph $G$ is a set $T \subseteq V(M)$
containing exactly one endpoint of each edge of $M$, an $M$-edge; this
endpoint is chosen.
The transversal is dominating if every vertex of $G$ unsaturated by $M$ has a neighbour
in $T$.
\end{definition}

\begin{proposition}[Reduction]\label{prop:reduction}
If a minimum maximal matching $M$ of a graph $G$ admits a dominating
transversal, then $\gamma(G) \le |M| = \edom(G)$.
\end{proposition}

\begin{proof}
Let $T$ be a dominating transversal of $M$, and let $Z = V(G) \setminus V(M)$ be
the unsaturated vertices. Every vertex of $V(M)$ either is a chosen endpoint or shares
its $M$-edge with a chosen endpoint, and is therefore dominated. As $T$ is dominating, every vertex of $Z$
has a neighbour in $T$. Since $V(M) \cup Z = V(G)$, $T$ is a dominating set of $G$ with
$|T| = |M|$, giving $\gamma(G) \le |M|$. As $M$ is a minimum maximal matching,
$|M| = \edom(G)$.
\end{proof}

\noindent
This reduces Conjecture~\ref{conj:baste} to finding a dominating transversal of a
minimum maximal matching. The set chosen at random in the proof
of~\cite[Theorem~1]{Baste2020} is a transversal. Computing $\edom(G)$ is equivalent to finding a largest independent set whose removal from $G$ leaves a perfect matching.

\begin{lemma}[Edge-domination formula]\label{lem:alphapm}
Let $G$ be a graph on $n$ vertices, and let $\alphapm(G)$ be the size of a largest
independent set $S$ such that $G - S$ has a perfect matching. Then
\[
  \edom(G) \;=\; \frac{n - \alphapm(G)}{2}.
\]
\end{lemma}

\begin{proof}
Let $M$ be a minimum maximal matching and $Z = V(G) \setminus V(M)$ the vertices $M$ leaves unsaturated.
By maximality, $Z$ is independent, and $M$ is a perfect matching of $G - Z$.
Hence $\alphapm(G) \ge |Z| = n - 2|M|$, and $|M| = \edom(G)$ gives $\edom(G) \ge (n - \alphapm(G))/2$.

\medskip\noindent
Conversely, let $S$ be an independent set of size $\alphapm(G)$ with a perfect
matching $M$ of $G - S$. The vertices unsaturated by $M$ are exactly those of
$S$. As $S$ is independent, no edge has both ends unsaturated, so $M$ is maximal,
and $\edom(G) \le |M| = (n - \alphapm(G))/2$. The two inequalities give the stated
equality.
\end{proof}

\noindent
Conjecture~\ref{conj:baste} is therefore equivalent to $2\gamma(G) +
\alphapm(G) \le n$ for every $\Delta$-regular $G$ with $\Delta \ge 1$. The
regularity-free case is that of independence number at most $n/3$.

\begin{corollary}[Regularity-free case]\label{cor:alpha}
If $G$ is a graph on $n$ vertices with $\alpha(G) \le n/3$, then $\gamma(G) \le
\edom(G)$.
\end{corollary}

\begin{proof}
A maximal independent set is dominating, and a largest one has size $\alpha(G)$,
so $\gamma(G) \le \alpha(G)$. A set achieving $\alphapm(G)$ is independent, so
$\alphapm(G) \le \alpha(G)$, and Lemma~\ref{lem:alphapm} gives $\edom(G) = (n -
\alphapm(G))/2$.
With $\alpha(G) \le n/3$,
\[
  \edom(G) \;\ge\; \frac{n - \alpha(G)}{2} \;\ge\; \frac{n}{3} \;\ge\; \alpha(G)
  \;\ge\; \gamma(G).
\]
\end{proof}

\noindent
Beyond that regularity-free case, a counting identity for $\Delta$-regular graphs
underlies the lower bound on $\edom$ of~\cite{BasteMMM2021} and shows when the
bound is tight.

\begin{proposition}[Counting identity]\label{prop:identity}
Let $G$ be a $\Delta$-regular graph on $n$ vertices, let $M$ be a maximal matching
of $G$, and let $q$ be the number of edges of $G$ with both ends in $V(M)$ that do
not lie in $M$. Then
\[
  (4\Delta - 2)\,|M| \;=\; \Delta n + 2q .
\]
Hence $\edom(G) \ge \Delta n/(4\Delta - 2)$, with equality if and only if $G$ has
a maximal matching that is induced.
\end{proposition}

\begin{proof}
The $2|M|$ vertices of $V(M)$ have degree $\Delta$, so the edges of $G$ meeting
$V(M)$ occupy $2\Delta|M|$ endpoints there. Each $M$-edge occupies two of those
endpoints, each of the $q$ other edges inside $V(M)$ occupies two, and each edge
from $V(M)$ to the unsaturated set $Z$ occupies one, so
\[
  2\Delta|M| \;=\; 2|M| + 2q + e\bigl(Z, V(M)\bigr).
\]
Maximality leaves $Z$ independent, so each of its vertices sends all $\Delta$ of
its edges into $V(M)$, giving $e(Z, V(M)) = \Delta|Z| = \Delta(n - 2|M|)$.
Substituting, $2\Delta|M| = 2|M| + 2q + \Delta(n - 2|M|)$, and collecting the
terms in $|M|$ gives $(4\Delta - 2)|M| = \Delta n + 2q$. As $\Delta$ and $n$ are
fixed, the identity makes $|M|$ grow with $q$, so a minimum maximal matching
minimises $q$. Since $q \ge 0$, this gives $\edom(G) \ge \Delta n/(4\Delta - 2)$,
with equality exactly when some maximal matching has $q = 0$, that is, is
induced~\cite[Lemma~3(i)]{BasteMMM2021}.
\end{proof}

\begin{definition}[At-risk vertex]\label{def:atrisk}
Let $M$ be a maximal matching of a $\Delta$-regular graph $G$ with unsaturated set
$Z$. A vertex $u \in Z$ is at risk if its $\Delta$ neighbours lie on $\Delta$
distinct $M$-edges.
\end{definition}

\noindent
If a vertex of $Z$ instead has two neighbours that form a single $M$-edge, every
transversal chooses one of them, a neighbour of that vertex, so the vertex is
dominated regardless; the proof of Theorem~1 in~\cite{Baste2020} notes this, and
the proof of Theorem~2 there sets these safe vertices aside to work with the
at-risk rest. Only at-risk vertices can be left undominated. A transversal
covers an at-risk vertex when a chosen endpoint is adjacent to it, and
leaves it uncovered otherwise. A transversal is dominating if and only if
it covers every at-risk vertex.

\section{An exact bound for degree at least seven}\label{sec:lll}

By the reduction, it remains to cover every at-risk vertex by a single
transversal. The random transversal itself is not new. One endpoint of each
edge of a minimum maximal matching is selected independently with probability
$\tfrac12$; the resulting set has size $\edom(G)$, and a given at-risk vertex
escapes it with probability $2^{-\Delta}$, so a first moment over the escaping
vertices gives the multiplicative bound~\cite{Baste2020}. What follows changes
only the last step. Instead of adding the expected number of uncovered vertices
to the transversal, the Local Lemma is applied to show that with positive
probability none escapes, which removes the additive term and makes the bound
exact.

A random transversal covers all at-risk vertices with positive probability
when $\Delta$ is large, as two forms of the Lov\'asz Local Lemma show. The
symmetric form~\cite{ErdosLovasz1975}
states that if each event of a family has probability at most $p$, each is
mutually independent of all but at most $d$ of the others, and $e\,p\,(d+1) \le 1$,
then with positive probability none of them occurs. The lopsided
form~\cite{ErdosSpencer1991} replaces mutual independence with a weaker
lopsidependency condition, a form of negative dependence. A graph $H$ on the events
is a lopsidependency graph if
\begin{equation}\label{eq:lopsi}
  \Pr\Bigl[A_i \,\Big|\, \textstyle\bigwedge_{j \in S} \overline{A_j}\Bigr]
  \;\le\; \Pr[A_i]
  \qquad\text{for every } i \text{ and every } S \subseteq V(H)\setminus N_H[i].
\end{equation}
The same conclusion then holds whenever $H$ has maximum degree $d$ and
$e\,p\,(d+1) \le 1$.

\begin{theorem}[Local Lemma bound]\label{thm:lll}
Every $\Delta$-regular graph $G$ with $\Delta \ge 7$ satisfies
$\gamma(G) \le \edom(G)$.
\end{theorem}

\begin{proof}
A symmetric application of the Local Lemma
reaches $\Delta \ge 9$, and the lopsided form, which keeps only the conflicting
dependencies and discards the agreeing ones, reaches $\Delta \ge 7$.

\medskip\noindent
\emph{Setup.} Let $M$ be a minimum maximal matching of $G$, so $|M| = \edom(G)$. Choose one
endpoint of each $M$-edge independently and uniformly at random; this is a
random transversal $T$, encoded by independent variables $X_e$, one per $M$-edge,
each uniform on the two endpoints of $e$. For an at-risk vertex $u$, let $A_u$ be
the event that $u$ is uncovered, that is, that none of its $\Delta$ neighbours is
chosen. These neighbours lie on $\Delta$ distinct $M$-edges
(Definition~\ref{def:atrisk}), so $A_u$ is an
intersection of $\Delta$ independent events of probability $\tfrac12$, giving
\[
  \Pr[A_u] = 2^{-\Delta}.
\]
If no event $A_u$ occurs, then $T$ covers every at-risk vertex, hence is
dominating, and Proposition~\ref{prop:reduction} gives
$\gamma(G) \le |M| = \edom(G)$. It remains to show
$\Pr[\bigwedge_u \overline{A_u}] > 0$.

\medskip\noindent
Each $A_u$ is an atom of the $\Delta$ variables $\{X_e : e \text{ incident to a
neighbour of } u\}$, fixing each to leave that neighbour unchosen. Two events
$A_u, A_v$ share a variable $X_e$ exactly when each of $u$ and $v$ has a neighbour
among the endpoints of $e$.

\medskip\noindent
\emph{Symmetric bound.} Every shared-variable dependency is counted. The event
$A_u$ is mutually independent of all $A_v$ sharing no variable with it. For a fixed $M$-edge
$e = \{x, x'\}$ with $u \sim x$, an at-risk $v \neq u$
sharing $X_e$ is a neighbour of $x$ or of $x'$ and, being unsaturated, is
neither endpoint. The
neighbours of $x$ besides $u$ and $x'$ number $\Delta-2$, and those of $x'$
besides $x$ number $\Delta-1$, so at most $2\Delta-3$ per edge and
$\Delta(2\Delta-3)$ in total. The
symmetric Local Lemma applies when $e\,2^{-\Delta}\bigl(\Delta(2\Delta-3)+1\bigr)
\le 1$, which holds for $\Delta \ge 9$, the left side being $\approx 0.72$ at
$\Delta = 9$ and $\approx 1.11$ at $\Delta = 8$.

\medskip\noindent
\emph{Lopsided bound.} Counting only the conflicting dependencies improves the
threshold. Let $H$ join
$A_u \sim A_v$ exactly when they conflict, sharing a variable $X_e$ on which they
demand different endpoints. This conflict graph is a lopsidependency graph by the
standard argument for atomic events. Fix a set $S$ of events none joined to
$A_u$; since $\Pr[A_u] > 0$, condition~\eqref{eq:lopsi} is equivalent to
$\Pr[B \mid A_u] \le \Pr[B]$, where $B = \bigwedge_{v \in S}\overline{A_v}$. Each
such $A_v$ agrees with $A_u$ on every shared variable, so conditioning on
$A_u$ leaves $A_v$ as a weaker event $A_v' \supseteq A_v$ on the remaining variables,
independent of $A_u$; equivalently, $\overline{A_v'} \subseteq \overline{A_v}$. Hence
\[
  \Pr[B \mid A_u] = \Pr\Bigl[\textstyle\bigwedge_{v\in S}\overline{A_v'}\Bigr]
  \le \Pr[B].
\]

\medskip\noindent
In $H$, $A_u$ is joined, for each $M$-edge $e = \{x,x'\}$ with
$u \sim x$, to the at-risk vertices adjacent to the opposite endpoint $x'$, which
demand $X_e = x$ where $A_u$ demands $x'$. The partner $x$ is saturated, hence
not at-risk, so the joined events lie among the other $\Delta - 1$ neighbours of
$x'$, at most $\Delta - 1$ per edge, and $H$
has maximum degree at most $\Delta(\Delta-1)$. The lopsided
Local Lemma applies when $e\,2^{-\Delta}\bigl(\Delta(\Delta-1)+1\bigr) \le 1$,
which holds for $\Delta \ge 7$, the left side being $\approx 0.91$ at
$\Delta = 7$ and $\approx 1.32$ at $\Delta = 6$. Hence
$\Pr[\bigwedge_u \overline{A_u}] > 0$, completing the proof.
\end{proof}

\begin{remark}
Theorem~\ref{thm:lll} proves the weaker of the two inequalities of
Section~\ref{sec:intro}, and the bounds above do not by themselves yield the
stronger. Certifying $\ind \le \edom$ the same way would require the random
transversal to be independent as well as dominating, and that constraint is not
locally sparse. Every maximal matching of $K_{r,r}$ is perfect, and a transversal
of one is an independent set exactly when all $r$ selections fall on the same
side, an event of probability $2^{1-r}$. Failure of independence is therefore
typical rather than rare, which is the regime the Local Lemma cannot treat. This
is not an impossibility, since $\ind(K_{r,r}) = \edom(K_{r,r}) = r$, and either side of
the bipartition is an independent dominating set of that size. The two
inequalities are nonetheless distinct inside the range settled above. Since
$\gamma(K_{r,r}) = 2$, the weaker inequality carries slack $r - 2$ at every
$\Delta = r \ge 7$ while the stronger stays tight.
\end{remark}

\begin{remark}
The degree threshold seven is intrinsic to the method. The lopsided criterion
clears at $\Delta \ge 7$, and on the full dependency graph, of
degree $\Delta(2\Delta-3)$, both the symmetric criterion and the optimal
criterion~\cite{Shearer1985} clear only at $\Delta \ge 9$; all three thresholds
are recorded in the accompanying code. None reaches the remaining degrees
$\Delta \in \{3,4,5,6\}$, where a dependency degree $d$ of order $\Delta^2$ makes
$e\,2^{-\Delta}(d+1)$ exceed $1$. Even the optimal criterion, computed on the
sparser lopsidependency graph of an explicit connected cubic graph, is
unsatisfied at $p = 2^{-3}$. Two sharper criteria fare no better. The
orderability improvement~\cite{Harris2016} at $k = \Delta$ and variable-degree
$2(\Delta-1)$ is unsatisfied for every $\Delta \le 6$, and the intersection
LLL~\cite{HeLiSun2023} adds to $p$ a correction of order $2^{-4\Delta}$, far
below the gap needed to close $\Delta = 6$. A different argument is required.
\end{remark}

\section{A local-search reformulation for cubic graphs}\label{sec:cubic}

Throughout this section $G$ is cubic, so $\Delta = 3$, and $M$ is a maximal matching.
By Proposition~\ref{prop:reduction}, Conjecture~\ref{conj:baste} for $G$ follows
once a minimum maximal matching admits a dominating transversal. Whether every
maximal matching admits one, a stronger property than the reduction needs, is
studied here by local search on the following potential.

The potential depends on a transversal via the multiplicities of the at-risk
vertices. For an at-risk vertex $u$, let
$\mathrm{mult}(u)$ be the
number of chosen endpoints adjacent to $u$; then $u$ is uncovered exactly when
$\mathrm{mult}(u) = 0$, and $\mathrm{mult}(u) \in \{0,1,2,3\}$ as $u$ has three
neighbours on distinct $M$-edges. Let $\beta$ be
the number of uncovered at-risk vertices and $P_1$ the number with
$\mathrm{mult} = 1$, and set
\[
  \Pot \;=\; 4\beta + P_1 .
\]
The weights grade an at-risk vertex by its exposure, giving $4$ when
uncovered, $1$ when a single chosen neighbour covers it, and $0$ when two or
three do. A
flip of an $M$-edge exchanges its chosen endpoint for the other. Since
$\Pot$ is a non-negative integer, every sequence of flips each strictly decreasing
$\Pot$ terminates; whether it can terminate only at $\beta = 0$ is examined below.

\begin{proposition}[Per-edge characterisation]\label{prop:peredge}
For an $M$-edge $e$ with chosen endpoint $c$ and unchosen endpoint $\bar c$, let
$a_e, b_e$ count the at-risk neighbours of $\bar c$ with $\mathrm{mult} = 0$ and
$1$, and $p_e, q_e$ those of $c$ with $\mathrm{mult} = 1$ and $2$. Then flipping
$e$ changes $\Pot$ by exactly
$(3p_e + q_e) - (3a_e + b_e)$. Hence a transversal is a local minimum of
$\Pot$ under flips if and only if $3p_e + q_e \ge 3a_e + b_e$ for every $M$-edge
$e$.
\end{proposition}

\begin{proof}
Flipping $e$ increments $\mathrm{mult}$ on the at-risk neighbours of $\bar c$ and
decrements it on the at-risk neighbours of $c$; no at-risk vertex is adjacent to
both endpoints of $e$, since that would place two of its neighbours on the one
edge $e$. Writing $\Pot = \sum_u
w(\mathrm{mult}(u))$ with weights $w(0) = 4$, $w(1) = 1$, $w(2) = w(3) = 0$, an
increment $0\!\to\!1$ contributes $w(1)-w(0) = -3$, an increment $1\!\to\!2$
contributes $-1$, and $2\!\to\!3$ contributes $0$; symmetrically the decrements
$1\!\to\!0$, $2\!\to\!1$, $3\!\to\!2$ contribute $+3$, $+1$, $0$. As $\bar c$ is
unchosen, its at-risk neighbours have $\mathrm{mult} \le 2$, so its mult-$0$ and
mult-$1$ neighbours, $a_e$ and $b_e$, change $\Pot$ by $-3a_e - b_e$; as $c$ is
chosen, its at-risk neighbours have $\mathrm{mult} \ge 1$, so its mult-$1$ and
mult-$2$ neighbours, $p_e$ and $q_e$, change it by $3p_e + q_e$. Summing gives the
change $(3p_e + q_e) - (3a_e + b_e)$.
\end{proof}

\begin{remark}
The characterisation of Proposition~\ref{prop:peredge} was verified against direct
recomputation on all $52{,}044$ transversals with $\beta \ge 1$ of all cubic
graphs on $n \le 12$ vertices, with no discrepancy.
\end{remark}

\begin{corollary}[Initial case]\label{cor:case1}
Let $u$ be an uncovered at-risk vertex and $e$ an $M$-edge incident to a neighbour
of $u$, with chosen endpoint $c$. If $c$ has no at-risk neighbour of $\mathrm{mult}
= 1$, then flipping $e$ strictly decreases $\Pot$.
\end{corollary}

\begin{proof}
By Proposition~\ref{prop:peredge}, flipping $e$ changes $\Pot$ by
$(3p_e + q_e) - (3a_e + b_e)$. As $u$ is uncovered, the neighbour of $u$ on $e$ is
unchosen and so equals $\bar c$, giving $a_e \ge 1$. The hypothesis is $p_e = 0$,
and $c$ has only two neighbours besides its partner $\bar c$, so $q_e \le 2$. The
change is thus at most $q_e - 3a_e \le 2 - 3 = -1 < 0$.
\end{proof}

\begin{remark}
Beyond this first case, an exhaustive check over all cubic graphs on $n \le 16$
vertices, all their maximal matchings, and all transversals found that every
transversal with $\beta \ge 1$ admits a flip strictly decreasing $\Pot$, so none
is a $\Pot$-local minimum, the largest case $n = 16$ accounting
for $54{,}715{,}529$ such transversals.
\end{remark}

\section{A barrier to fractional arguments}\label{sec:barrier}

Discharging and Hall-type arguments on the per-edge inequalities of
Proposition~\ref{prop:peredge} cannot force the uncovered-vertex density to zero, as the witness below
shows. For an $M$-edge of a cubic graph,
the at-risk neighbours of the chosen endpoint have multiplicities in $\{1,2,3\}$
and those of the unchosen endpoint in $\{0,1,2\}$, at most two on each side; the
type of the edge records these two multisets.

A $\Pot$-local minimum induces a density over types, namely the fraction of
$M$-edges of each type, and this density satisfies two families of linear
constraints. First, the per-edge inequality of Proposition~\ref{prop:peredge}
holds on every edge,
\[
  3p + q \;\ge\; 3a + b.
\]
Second, a counting constraint relates the incidences. A mult-$1$ at-risk vertex is a chosen
neighbour on one incident edge and an unchosen neighbour on two, so its unchosen
incidences are twice its chosen ones, and a mult-$2$ vertex is chosen on two
edges and unchosen on one. The density of every genuine $\Pot$-local minimum
satisfies these constraints, so the resulting linear program is a relaxation,
its uncovered-vertex density $\rho$ the fractional analogue of the count $\beta$ in
Section~\ref{sec:cubic}. If the relaxation forced $\rho = 0$, no cubic graph
could be a $\Pot$-local minimum with $\beta \ge 1$.

\begin{proposition}[Fractional barrier]\label{prop:barrier}
The relaxation above is feasible with $\rho > 0$. Hence no linear
discharging over edge-multiplicity profiles together with the counting
constraints can force $\rho = 0$.
\end{proposition}

\begin{proof}
The witness places density $\tfrac14$ on the type with chosen multiset $\emptyset$
and unchosen $\{2,2\}$, $\tfrac14$ on chosen $\{1,1\}$ and unchosen $\{0,0\}$, and
$\tfrac12$ on chosen $\{2,2\}$ and unchosen $\{1,1\}$. Each
type is tight, namely $0=0$, $6=6$, and $2=2$, so every per-edge
inequality holds. Both counting constraints hold. Each incidence kind occurs in a
single type, so its density is the two neighbours of that kind per edge times
that type's density. For mult-$1$, the unchosen incidence
density $2\cdot\tfrac12 = 1$ is twice the chosen density $2\cdot\tfrac14 =
\tfrac12$; for mult-$2$, the chosen density $2\cdot\tfrac12 = 1$ is twice the
unchosen density $2\cdot\tfrac14 = \tfrac12$. The uncovered incidence
density is $2\cdot\tfrac14 = \tfrac12$, and each uncovered vertex carries three
unchosen incidences, so the uncovered-vertex density is $\rho = \tfrac16 > 0$.
Since every linear inequality valid for the relaxation evaluates non-negatively on
this feasible point, none can force $\rho = 0$.
\end{proof}

The barrier is not special to the cubic case. For general $\Delta$ the same
reduction grades an at-risk vertex by the weight $\Delta + 1$ when uncovered and
$1$ when covered once, giving the potential $\Pot_\Delta = (\Delta + 1)\beta +
P_1$ with $P_1$ the number of mult-$1$ vertices; at $\Delta = 3$ this is the
potential $\Pot$ of Section~\ref{sec:cubic}. The per-edge change, by the weight
bookkeeping of Proposition~\ref{prop:peredge}, is $(\Delta p + q) - (\Delta a +
b)$, with $p, q$ counting the chosen endpoint's at-risk neighbours of mult $1$
and $2$ and $a, b$ the unchosen endpoint's of mult $0$ and $1$, so the relaxation
carries over with filter $\Delta p + q \ge \Delta a + b$.
The counting constraint likewise generalises. A mult-$j$ at-risk vertex is chosen
on $j$ of its $\Delta$ edges and unchosen on $\Delta - j$, so its chosen and
unchosen incidences stand in ratio $j : (\Delta - j)$.

\begin{proposition}[Barrier at every regular degree]\label{prop:barrier-delta}
For every $\Delta \ge 3$ the relaxation with potential $\Pot_\Delta$ is feasible
with
\[
  \rho(\Delta) \;=\; \frac{2(\Delta - 1)}{\Delta\,(\Delta^{2} - \Delta + 2)}
  \;>\; 0 .
\]
Hence, at every regular degree, no linear discharging with this potential forces
$\rho = 0$.
\end{proposition}

\begin{proof}
Three edge types carry the witness, with $2^{j}$ denoting $j$ entries equal to
$2$. Type $A$ has chosen multiset $\{2\}$ and unchosen $\{1, 2^{\Delta - 2}\}$;
type $B$ has chosen $\{1, 2\}$ and unchosen $\{0, 1, 2^{\Delta - 3}\}$; type $C$
has chosen $\{2, 2\}$ and unchosen $\{1, 1, 2^{\Delta - 3}\}$. Each type is tight,
the identity $\Delta p + q = \Delta a + b$ reading $1 = 1$, $\Delta + 1 = \Delta +
1$, and $2 = 2$, so every per-edge inequality holds. Set
\[
  x_A = \frac{2}{\Delta - 1}\,x_B, \qquad
  x_B = \frac{2(\Delta - 1)}{\Delta^{2} - \Delta + 2}, \qquad
  x_C = \frac{\Delta(\Delta - 3)}{2(\Delta - 1)}\,x_B.
\]
These densities are non-negative, with $x_C = 0$ at $\Delta = 3$, and normalise
to $x_A + x_B + x_C = 1$. The counting constraints, non-trivial only at
multiplicities $1$ and $2$, now involve incidence densities summing over several
types, unlike the single terms of the cubic witness. A mult-$0$ vertex is an unchosen neighbour
on each of its $\Delta$ edges, and only type $B$ carries one, so the uncovered
incidence density is $x_B$ and $\rho(\Delta) = x_B / \Delta$, the stated value.
At $\Delta = 3$ it equals $\tfrac16$, recovering the cubic value of
Proposition~\ref{prop:barrier} through a different feasible point. The counting
constraints and the normalisation, identities in $\Delta$, are verified in the
accompanying code.
\end{proof}

For $\Delta \ge 7$ the conjecture holds by Theorem~\ref{thm:lll} regardless, so
the barrier bounds the discharging method, not the conjecture. The degrees
$\{3, 4, 5, 6\}$, beyond the reach of the Local Lemma, yield to neither approach.

Strengthening the linear relaxation to a degree-two
sum-of-squares~\cite{Lasserre2001,Parrilo2003} at the family level does no better
on the cubic instance. The
lift over pairs of edge types adjoins a moment matrix constrained to be positive
semidefinite, and on the witness above that matrix decomposes as a
non-negative combination of rank-one positive-semidefinite terms, one per pair of
types, so the constraint holds automatically, leaving the witness feasible. Only a
per-instance lift, whose moments encode the distance-two non-adjacency of the
at-risk vertices on specific edges, could carry positive-semidefinite content that
tightens the relaxation, and no such lift is a family-level argument. The
decomposition is verified for the cubic witness in the accompanying code; whether
it persists at every degree is not addressed here.

\section{A satisfiability reformulation}\label{sec:sat}

The transversal search of Section~\ref{sec:reduction} is, in the cubic case, an
instance of a restricted satisfiability problem at a classical threshold. Let $G$
be cubic and $M$ a maximal matching. Each edge $e \in M$ carries a
Boolean variable $x_e$ whose two truth values name the two endpoints of $e$, so a
truth assignment is a transversal. An at-risk vertex $u$ has neighbours on three
distinct $M$-edges $e_1, e_2, e_3$; let $\ell_i$ be the literal on $x_{e_i}$
that is true when the endpoint adjacent to $u$ is chosen. Then $u$ is covered
exactly when the clause $\ell_1 \vee \ell_2 \vee \ell_3$ is satisfied. Let
$\Phi(G,M)$ be the conjunction of these clauses over the at-risk vertices.

\begin{proposition}[Satisfiability reformulation]\label{prop:sat}
A maximal matching $M$ of a cubic graph $G$ admits a dominating transversal if
and only if $\Phi(G,M)$ is satisfiable. In particular, if $\Phi(G,M)$ is
satisfiable for a minimum maximal matching $M$, then $\gamma(G) \le \edom(G)$.
\end{proposition}

\begin{proof}
Under the naming, truth assignments are exactly transversals, and the clause of an
at-risk vertex is satisfied exactly when the transversal covers that vertex. A
transversal thus satisfies $\Phi$ exactly when it covers every at-risk vertex,
hence exactly when it is dominating, giving the equivalence. For a minimum
maximal matching, Proposition~\ref{prop:reduction} then gives the implication.
\end{proof}

\begin{lemma}[Structure of $\Phi$]\label{lem:sat}
Each clause of $\Phi(G,M)$ has three literals on distinct variables, and each
variable occurs at most twice positively and at most twice negatively, hence in
at most four clauses.
\end{lemma}

\begin{proof}
An at-risk vertex has three neighbours on three distinct $M$-edges
(Definition~\ref{def:atrisk}), so its clause has three literals on distinct
variables. The variable $x_e$ occurs for each at-risk neighbour of an endpoint of
$e$. An endpoint has three neighbours, its partner and two others, and the
partner is saturated, hence not at-risk, so at most two of them are at-risk. The
two endpoints of $e$ give the two polarities of $x_e$, so $x_e$ occurs at most
twice in each polarity.
\end{proof}

\begin{proposition}[Few at-risk vertices]\label{prop:fewatrisk}
A minimum maximal matching of a cubic graph $G$ on $n$ vertices leaves at most
$\lfloor 2n/5 \rfloor$ at-risk vertices. If it leaves at most nineteen, then
$\gamma(G) \le \edom(G)$; in particular $\gamma(G) \le \edom(G)$ for every cubic
graph on at most forty-eight vertices.
\end{proposition}

\begin{proof}
Let $M$ be a minimum maximal matching of $G$ and $Z$ the set of vertices it leaves
unsaturated, an independent set by maximality. Proposition~\ref{prop:identity}
gives $|M| \ge \tfrac{3}{10}\,n$, so $|Z| = n - 2|M| \le \tfrac{2}{5}\,n$.
The at-risk vertices lie in $Z$, so their number $r$ is at most
$\lfloor 2n/5 \rfloor$, which is at most nineteen for $n \le 48$.

Suppose $r \le 19$. By Proposition~\ref{prop:sat} it suffices that $\Phi(G,M)$ is
satisfiable. By Lemma~\ref{lem:sat} each clause has three literals on distinct
variables and each variable occurs at most twice in either polarity, so
$\Phi(G,M)$ is a $(3,2,2)$-formula, a $3$-CNF in which every variable has at
most two positive and two negative occurrences. A smallest unsatisfiable such
formula has twenty clauses~\cite{ZhangPeitlSzeider2024}, so every one with at
most nineteen clauses is satisfiable. As $\Phi(G,M)$ has $r \le 19$ clauses it is
satisfiable, and Proposition~\ref{prop:sat} gives $\gamma(G) \le \edom(G)$.
\end{proof}

\noindent
Following~\cite{Tovey1984}, write $(k,s)$-SAT for the satisfiability
instances with exactly $k$ variables in each clause and at most $s$ occurrences
of each variable. By Lemma~\ref{lem:sat}, $\Phi(G,M)$ lies in $(3,4)$-SAT, at the
satisfiability boundary. Every $(3,3)$-SAT instance is satisfiable, whereas
$(3,4)$-SAT is NP-complete. Neither the occurrence bound nor the polarity balance
forces satisfiability. Even the tight case of Lemma~\ref{lem:sat}, in which each
literal occurs exactly twice, admits unsatisfiable instances~\cite{Berman2003}.
The smallest such instance has twenty clauses~\cite{ZhangPeitlSzeider2024}, which
is exactly what Proposition~\ref{prop:fewatrisk} exploits. The occurrence
structure alone settles every cubic graph on at most forty-eight vertices. A
weaker but self-contained threshold needs no exhaustive input, since a first
moment bound already clears $r \le 7$. Under a uniform random assignment each
clause fails with probability $2^{-3}$, so the expected number of failures is
$r/8 < 1$ and some assignment satisfies $\Phi$. The graphs left open are those
whose minimum maximal matching leaves twenty or more at-risk vertices, the first
at $n = 50$. At that order $\Phi$ can be unsatisfiable, as the following
construction shows.

\begin{theorem}[Integral obstruction]\label{thm:cex}
There is a connected cubic graph $G$ on fifty vertices whose unique minimum
maximal matching admits no dominating transversal, yet whose domination and edge
domination numbers satisfy $\gamma(G) = 14 < 15 = \edom(G)$.
\end{theorem}

\begin{proof}
Take a minimum unsatisfiable $3$-CNF $f$ in which each literal occurs exactly
twice~\cite{Berman2003}; the least such formula has $\mu(3,2,2) = 20$
clauses~\cite{ZhangPeitlSzeider2024}, here on fifteen variables. Build $G$ from the incidence of $f$. The matching $M$ has one edge per variable, joining that variable's two polarity
vertices, and each clause, whose three literals lie on distinct variables, gives
a further vertex joined to the endpoint named by each of them. An endpoint has
degree $1 + 2 = 3$, its $M$-edge
and its two occurrences, and a clause vertex has degree three, so $G$ is cubic
and simple. It is connected, since a disconnected incidence would split $f$ into two
smaller formulas of the same class, one unsatisfiable, contradicting the
minimality of $f$. The twenty clause vertices are unsaturated and pairwise
non-adjacent, so $M$, the fifteen variable edges, is maximal; each polarity vertex
meets only its partner and clause vertices, so $M$ is induced and
Proposition~\ref{prop:identity} makes $|M| = 15 = \tfrac{3}{10}\,n$ minimum. By construction $\Phi(G,M) = f$, which is
unsatisfiable, so no transversal of $M$ dominates the unsaturated set; equivalently
every transversal has $\beta \ge 1$. As $\Pot$ is a non-negative integer, greedy
descent halts at a $\Pot$-local minimum, necessarily with $\beta \ge 1$. That $M$
is the unique minimum maximal matching and that $\gamma(G) = 14$ are certified in
Appendix~\ref{app:cert}.
\end{proof}

\noindent
The feasibility of the relaxation in Proposition~\ref{prop:barrier} showed that no
linear argument over it closes the cubic case; Theorem~\ref{thm:cex} shows more. For an
actual cubic graph, $\Pot$ has a local minimum with $\beta \ge 1$, so even the
exact local-search descent, not only its relaxation, fails. Its at-risk vertices
are pairwise non-adjacent, the distance-two non-adjacency that
Section~\ref{sec:barrier} found beyond any family-level relaxation, so even that
condition does not close the case. Since
$M$ is the only minimum maximal matching and has no dominating transversal, the
reduction of Section~\ref{sec:reduction} cannot prove $\gamma(G) \le \edom(G)$
here; the dominating set of size $14$ is not a transversal of $M$. Every proof of
the cubic case must reason beyond transversals of minimum maximal matchings.

\section{Sharpness and the open range}\label{sec:sharp}

At this same order the inequality is also tight, and it stays tight on an infinite
family of cubic graphs, the generalized Petersen graphs. The graph
$GP(k,2)$ has vertices $u_0, \dots, u_{k-1}$ and $v_0, \dots, v_{k-1}$, with
edges $u_i u_{i+1}$, $u_i v_i$ and $v_i v_{i+2}$, indices modulo $k$; it is cubic
for $k \ge 5$.

\begin{proposition}[Sharpness]\label{prop:sharp}
Every $GP(k,2)$ with $k \ge 5$ contains an induced claw. For every $k$ divisible
by $5$,
\[
  \gamma(GP(k,2)) \;=\; \edom(GP(k,2)) \;=\; \tfrac{3}{5}k.
\]
\end{proposition}

\begin{proof}
The neighbours of $u_i$ are $u_{i-1}$, $u_{i+1}$ and $v_i$. For $k \ge 5$ the
outer vertices $u_{i-1}$ and $u_{i+1}$ are non-adjacent, and $v_i$ meets only
$v_{i \pm 2}$ and $u_i$, so the three are pairwise non-adjacent and, with their
common neighbour $u_i$, induce a claw.

Let $k = 5m$. The domination number of $GP(k,2)$ is
$k - \lfloor k/5 \rfloor - \lfloor (k+2)/5
\rfloor$~\cite{FuYangJiang2009}, which equals $3k/5$ when $5$ divides $k$. On
$n = 2k$ vertices the counting bound $\edom \ge \Delta n/(4\Delta -
2)$ of Proposition~\ref{prop:identity} reads $\edom \ge 3k/5$. For the reverse inequality take,
for $j = 0, \dots, m-1$, the three edges
\[
  u_{5j}u_{5j+1}, \qquad u_{5j+3}v_{5j+3}, \qquad v_{5j+2}v_{5j+4} .
\]
No two of the resulting $3m$ edges share a vertex, and the unsaturated vertices
are $u_{5j+2}$, $u_{5j+4}$, $v_{5j}$ and $v_{5j+1}$. Every neighbour of each is
saturated, as $u_{5j+2}$ meets $u_{5j+1}$, $u_{5j+3}$ and $v_{5j+2}$; $u_{5j+4}$
meets $u_{5j+3}$, $u_{5j+5}$ and $v_{5j+4}$; $v_{5j}$ meets $u_{5j}$, $v_{5j+2}$
and $v_{5j-2}$; and $v_{5j+1}$ meets $u_{5j+1}$, $v_{5j+3}$ and $v_{5j-1}$. No
edge joins two unsaturated vertices, so this matching is maximal and $\edom \le
3m = 3k/5$, meeting the lower bound. With the domination number also $3k/5$ this
gives $\gamma = \edom = 3k/5$.
\end{proof}

\begin{remark}
A computer check finds $\gamma = \edom = \lceil 3k/5 \rceil$ for every $5 \le k \le 40$
with $k \not\equiv 3 \pmod 5$, and strict inequality at the remaining $k$ in that range.
\end{remark}

\noindent
Equality thus holds on the infinite family and at every verified $k \not\equiv 3
\pmod 5$. The family gives $GP(25,2)$ on $n = 50$ vertices and the finite check
$GP(26,2)$ on $n = 52$, both beyond the reach of Proposition~\ref{prop:fewatrisk}. Two
consequences bear on the cubic case. First, the graphs above contain claws, so
they lie outside the claw-free class whose equality cases are characterised
in~\cite{PanPanTie2025}. Second, every proof must be exactly tight on infinitely
many graphs, so no argument carrying slack can close the case. Sharpness is not
confined to $\Delta = 3$. Among the circulant graphs examined equality also
occurs at $\Delta = 4$, always with $\gamma = \edom = n/3$, above the counting
floor $2n/7$ there, while none was found at $\Delta = 5$ or $\Delta = 6$.

\medskip\noindent
The floor $\edom \ge \Delta n/(4\Delta-2)$ also locates what remains of the open
range. At sufficiently large girth a $\Delta$-regular graph has an independent
dominating set of size at most $c_\Delta n$, where $c_3 = 0.27942$, $c_4 =
0.24399$, $c_5 = 0.21852$ and $c_6 = 0.19895$~\cite{HoppenWormald2018}; at
$\Delta = 3$ this improves the $0.299871\,n$ of~\cite{KralSkodaVolec2012}. The
constants were obtained for random regular graphs
in~\cite{DuckworthWormald2002,DuckworthWormald2006} and carried to large
girth in~\cite{HoppenWormald2018}. Since such a set is dominating, the domination
number obeys the same bound.
Each constant lies below the threshold $\Delta/(4\Delta-2)$ that the combination
of Section~\ref{sec:intro} requires, by $0.021$, $0.042$, $0.059$ and $0.074$
respectively, so at every open degree the independent-domination ceiling drops
below the floor with room to spare and the conjecture holds for the graphs
concerned. The crossing gives more than
Conjecture~\ref{conj:baste}. The bound is on the independent domination number
itself, so
\[
  \ind(G) \;\le\; c_\Delta n \;<\; \frac{\Delta}{4\Delta-2}\,n \;\le\; \edom(G)
  \qquad (3 \le \Delta \le 6).
\]
Hence a $\Delta$-regular graph of sufficiently large girth satisfies the stronger
inequality of~\cite{Davila2025reverie} as well, which is open for every $\Delta \ge
3$. With Theorem~\ref{thm:lll} settling $\Delta \ge 7$, what is left of
Conjecture~\ref{conj:baste} is small girth alone. Claw-free and fork-free graphs
are already settled~\cite{Civan2023,ManiyaPradhan2024}, so the residue contains an
induced fork. No explicit girth is available at which
the crossing starts, the bounds of~\cite{HoppenWormald2018} being stated for
sufficiently large girth throughout,
so what the crossing gives is a confinement
of the open part to small girth, not a family one could
exhibit.

\section{The reformulation at every degree}\label{sec:degrees}

The reformulation is not special to cubic graphs. For a $\Delta$-regular graph
the clause of an at-risk vertex has $\Delta$ literals, one for each of its
$\Delta$ neighbours, which lie on distinct $M$-edges, and the argument of
Proposition~\ref{prop:sat} is degree-independent. A minimum maximal matching admits a dominating transversal
exactly when the resulting $\Delta$-uniform formula is satisfiable. A first-moment bound then
settles every degree on a bounded number of vertices.

\begin{proposition}[Bounded orders in every degree]\label{prop:bddorder}
Let $G$ be a $\Delta$-regular graph on $n$ vertices. A minimum maximal matching
leaves at most $\lfloor (\Delta-1)n/(2\Delta-1) \rfloor$ at-risk vertices, and if
it leaves at most $2^\Delta - 1$, then $\gamma(G) \le \edom(G)$. In
particular $\gamma(G) \le \edom(G)$ for every $\Delta$-regular graph on at most
$37$, $70$, and $140$ vertices when $\Delta = 4$, $5$, and $6$.
\end{proposition}

\begin{proof}
Let $M$ be a minimum maximal matching and $Z$ its unsaturated set, independent by
maximality. Proposition~\ref{prop:identity} gives $|M| \ge \Delta n/(4\Delta-2)$,
so $|Z| = n - 2|M| \le n(\Delta-1)/(2\Delta-1)$ and the number $r$ of at-risk
vertices, which lie
in $Z$, is at most $\lfloor (\Delta-1)n/(2\Delta-1) \rfloor$. The
$\Delta$-uniform formula $\Phi(G,M)$ has $r$ clauses; under a uniform random
assignment each is falsified with probability $2^{-\Delta}$, so the expected
number falsified is $r\,2^{-\Delta} < 1$ once $r \le 2^\Delta - 1$, and some
assignment satisfies all of them. The reformulation and
Proposition~\ref{prop:reduction} then give $\gamma(G) \le \edom(G)$. For $\Delta = 4, 5, 6$ the threshold $2^\Delta - 1$ is $15, 31, 63$, not exceeded
by $\lfloor (\Delta-1)n/(2\Delta-1) \rfloor$ for $n$ at most $37, 71, 140$; a
$5$-regular graph has even order, so $\Delta = 5$ covers $n \le 70$.
\end{proof}

\noindent
Each variable is an $M$-edge with two polarities, one per endpoint, and the
literal for an endpoint occurs once for each at-risk neighbour it covers. An
endpoint has $\Delta - 1$ neighbours besides its saturated partner, so at most
$\Delta - 1$ are at-risk, each polarity occurs at most $\Delta - 1$ times, and
$\Phi(G,M)$ is a $(\Delta, \Delta-1, \Delta-1)$-formula. As with $\mu(3,2,2) = 20$
in the cubic case, the minimum size of an unsatisfiable one would sharpen these
ranges, and a
$\Delta$-regular graph realising such a formula would, as in
Theorem~\ref{thm:cex}, give a minimum maximal matching with no dominating
transversal; whether such a formula exists for $\Delta \in \{4, 5, 6\}$ is open. At $\Delta
= 4$ and $\Delta = 5$ the occurrence count alone cannot settle the question. A
$(4,3,3)$-formula allows a variable six occurrences and a $(5,4,4)$-formula
eight, while unsatisfiable $(4,5)$-SAT and $(5,8)$-SAT instances
exist~\cite{HoorySzeider2005}. At
$\Delta = 6$ the count decides nothing either way. A $(6,5,5)$-formula allows ten
occurrences, and the largest $s$ for which every $(6,s)$-SAT instance is
satisfiable is known only to lie between seven and
eleven~\cite{HoorySzeider2005}. Only the
uniform clause width and the counting bound are used above.

\medskip\noindent
At the next three degrees the same occurrence count becomes decisive. A $(7,6,6)$-formula allows a
variable twelve occurrences, an $(8,7,7)$-formula fourteen and a $(9,8,8)$-formula
sixteen, whereas every $(7,13)$-, $(8,24)$- and $(9,41)$-SAT instance is
satisfiable~\cite{HoorySzeider2005}. Every formula $\Phi(G,M)$ arising at those
three degrees is therefore satisfiable, so Conjecture~\ref{conj:baste} holds
there by the reduction of Section~\ref{sec:reduction} together with those
bounds. The proof of Theorem~\ref{thm:lll} uses no such bound and settles
every $\Delta \ge 7$ in a single argument.

\section{Conclusion}\label{sec:open}

The reduction to a dominating transversal proves Conjecture~\ref{conj:baste} for
every $\Delta \ge 7$ through the Lov\'asz Local Lemma in a single argument, with
published occurrence thresholds giving an independent route at $\Delta \in
\{7,8,9\}$, leaving the degrees $\{3,4,5,6\}$. The middle degrees $\{4,5,6\}$, too small
for the Local Lemma and too large for the cubic case, remain open in general,
though Proposition~\ref{prop:bddorder} settles each on a bounded number of
vertices.

The cubic case
$\Delta = 3$, where the bound is sharp, is carried furthest. The reduction to a
dominating transversal of a minimum maximal matching, recast as satisfiability of
a balanced formula, settles every cubic graph on at most forty-eight vertices
through the least unsatisfiable size $\mu(3,2,2) = 20$
(Proposition~\ref{prop:fewatrisk}). At $n = 50$ the method meets its limit, and
two obstructions remain. The first is that the relaxation of
Section~\ref{sec:barrier} is feasible with uncovered-vertex density $\rho > 0$,
and Theorem~\ref{thm:cex} strengthens
this to the integers with a cubic graph whose unique minimum maximal matching has
no dominating transversal, so the transversal reduction and its local search both
fail. Proposition~\ref{prop:sharp} gives the second, the inequality tight on an
infinite family of cubic graphs containing claws, so no argument carrying slack
can close it. Conjecture~\ref{conj:baste} for cubic graphs remains open, and every
proof must be exactly tight on that family while reasoning beyond transversals of
minimum maximal matchings.

\appendix

\section{Certificate for the integral obstruction}\label{app:cert}

The graph of Theorem~\ref{thm:cex} is built from an explicit formula. Take $f$ to
be the twenty-clause $3$-CNF on the fifteen variables $x_1, \dots, x_{15}$ whose
clauses, in four blocks of five, are
\[
\begin{array}{@{}l@{\qquad}l@{}}
x_1 \vee x_4 \vee \bar{x}_6 & x_1 \vee x_7 \vee \bar{x}_9 \\
x_2 \vee \bar{x}_4 \vee x_5 & \bar{x}_2 \vee \bar{x}_7 \vee x_8 \\
x_2 \vee \bar{x}_5 \vee x_6 & \bar{x}_2 \vee \bar{x}_8 \vee x_9 \\
x_4 \vee x_5 \vee x_6 & x_7 \vee x_8 \vee x_9 \\
\bar{x}_4 \vee \bar{x}_5 \vee \bar{x}_6 & \bar{x}_7 \vee \bar{x}_8 \vee \bar{x}_9
\end{array}
\]
\[
\begin{array}{@{}l@{\qquad}l@{}}
\bar{x}_1 \vee x_{10} \vee \bar{x}_{12} & \bar{x}_1 \vee x_{13} \vee \bar{x}_{15} \\
x_3 \vee \bar{x}_{10} \vee x_{11} & \bar{x}_3 \vee \bar{x}_{13} \vee x_{14} \\
x_3 \vee \bar{x}_{11} \vee x_{12} & \bar{x}_3 \vee \bar{x}_{14} \vee x_{15} \\
x_{10} \vee x_{11} \vee x_{12} & x_{13} \vee x_{14} \vee x_{15} \\
\bar{x}_{10} \vee \bar{x}_{11} \vee \bar{x}_{12} & \bar{x}_{13} \vee \bar{x}_{14} \vee \bar{x}_{15} .
\end{array}
\]
The clauses are numbered $1$ to $20$ down the columns from left to right, the
first display before the second. Every literal occurs exactly twice, and $f$ is
unsatisfiable, least among such
formulas at twenty clauses~\cite{ZhangPeitlSzeider2024}. The incidence graph $G$
carries, for each variable $x_i$, two vertices $v_i$ and $\bar{v}_i$ joined by an
$M$-edge, and for each clause a vertex $w_j$ joined to $v_i$ at a positive literal
$x_i$ and to $\bar{v}_i$ at a negative literal $\bar{x}_i$. With $M = \{v_i
\bar{v}_i : 1 \le i \le 15\}$ this is the cubic connected graph on fifty vertices
of Theorem~\ref{thm:cex}.

A dominating set of order fourteen is
\[
\{v_1, \bar{v}_1, \bar{v}_2, v_3, \bar{v}_3, \bar{v}_4, v_5, v_6\} \cup
\{w_9, w_{10}, w_{14}, w_{15}, w_{19}, w_{20}\},
\]
the eight polarity vertices and six clause vertices, so $\gamma(G) \le 14$. This
set is not a transversal of $M$, meeting two of its fifteen edges in both
endpoints, four in one, and nine in neither. Exhaustive search finds no dominating
set of order thirteen and exactly one maximal matching of order fifteen, so
$\gamma(G) = 14$ and $M$ is the unique minimum maximal matching. Since $\Phi(G, M)
= f$ is unsatisfiable, no transversal of $M$ dominates $G$, and the domination
number is reached only off the matching, by the fourteen vertices above.

\bibliographystyle{alpha}
\bibliography{references}

\end{document}